\documentclass[12pt]{amsart}
\usepackage[top=1in, bottom=1in, left=1in, right=1in]{geometry}
\usepackage{amsfonts}
\usepackage{amsmath}
\usepackage{comment}
\usepackage{amssymb}
\usepackage{float}
\usepackage{graphicx} 
\usepackage{bbm}
\usepackage{comment} 
\usepackage{mathrsfs}
\numberwithin{equation}{section}
\usepackage{times}
\usepackage{tikz, pgfplots}
\pgfplotsset{compat=1.18}
\usetikzlibrary{positioning}
\usepackage{siunitx}
\usepackage{color}
\usetikzlibrary{angles}
\usetikzlibrary{arrows.meta}
\usepackage{comment}
\usepackage{xcolor}
\usepackage{mathtools}
\usepackage{bm}
\usepackage{esvect}
\usepackage{hyperref}

\hypersetup{
    colorlinks = true,
linkcolor={black},
urlcolor={blue},
citecolor={blue},    
urlcolor = {blue},
citebordercolor = {0.33 .58 0.33},
 linkbordercolor = {0.99 .28 0.23},
 breaklinks=true}
 
\newcommand{\R}{\mathbb{R}}

\newcommand{\Z}{\mathbb{Z}}

\newtheorem{thm}{Theorem}[section]

\newtheorem{prop}[thm]{Proposition}
\newtheorem{coro}[thm]{Corollary}
\newtheorem{lem}[thm]{Lemma}
\newtheorem{rem}[thm]{Remark}

\theoremstyle{remark}

\usepackage{geometry}
\title{Self-intersection Points of Billiard Trajectories in a Square with Small Pockets}
\author{Anji Dong, Colin Edsey, Shuta Iwai, Alexandru Zaharescu}

\address{
Anji Dong: Department of Mathematics,
University of Illinois Urbana-Champaign,
Altgeld Hall, 1409 W. Green Street,
Urbana, IL, 61801, USA}
\email{anjid2@illinois.edu}

\address{
Colin Edsey: Department of Mathematics,
University of Illinois Urbana-Champaign,
Altgeld Hall, 1409 W. Green Street,
Urbana, IL, 61801, USA}
\email{csedsey2@illinois.edu}

\address{
Shuta Iwai: Department of Mathematics,
University of Illinois Urbana-Champaign,
Altgeld Hall, 1409 W. Green Street,
Urbana, IL, 61801, USA}
\email{siwai2@illinois.edu}

\address{
Alexandru Zaharescu: Department of Mathematics,
University of Illinois Urbana-Champaign,
Altgeld Hall, 1409 W. Green Street,
Urbana, IL, 61801, USA}
\email{zaharesc@illinois.edu}

\begin{document}
\setcounter{tocdepth}{1}
\begin{abstract}
   We study the self-intersections of billiard trajectories in a square with small pockets of size $\varepsilon$ removed from its four corners. In particular, we establish an asymptotic formula  for the $r$-th moment of the number of self-intersection points for each positive integer $r$, as well as the distribution of such points.
\end{abstract}

\maketitle

\section{Introduction and main results}
Billiard systems in polygonal domains have been studied extensively over the past several decades. The behavior of such trajectories depends strongly on the arithmetic properties of their slopes. For classical results in this direction, see the works of Veech \cite{Veech1989, Veech1992}. These systems  provide a rich interplay between geometry, dynamics, and number theory. In particular, the study of billiard trajectories has been closely related to Farey sequences, which play an important role in Diophantine approximation and are also closely connected to the Riemann zeta function. Early works in this direction include, for example, those of Boca, Cobeli, and the fourth author \cite{BocaCobeliZaha2000}, as well as Bunimovich and Dettmann \cite{BunimovichDettmann}. More recent methods have established ergodic and statistical properties of the periodic two-dimensional Lorentz gas, which was introduced by Lorentz \cite{Lorentz1905} to study the dynamics of electrons in metals. We refer the reader to the works of Alkan, Ledoan, Selvakumaran, Boca, Gologan, and the fourth author \cite{AlkanLedoanZaha, Alkan2005, BocaGologanZaha, BocaGologanZaha2}, as well as P\`ene \cite{FranPene2014} for these and further developments. 

In the present paper, we consider a billiard table in the shape of a unit square, denoted by $\Omega$, with small triangular pockets of size $\varepsilon$ removed from the four corners, as shown in Figure \ref{fig:setting} below. A point-like particle is launched from the bottom-left corner with an initial angle  $\theta\in [0,\pi/2]$ measured from the horizontal axis. Our primary goal is to study the number and  distribution of self-intersection points of the trajectory before the particle enters one of the pockets. A self-intersection point occurs when the trajectory visits the same point in $\Omega$ at two distinct times; that is, there exist $t_1\neq t_2$ such that $P(t_1)=P(t_2)$, where $P(t)$ denotes the position of a particle at time $t$. 

To capture the arithmetic structure of the problem, we unfold the billiard table to $\R^2$. Observe that the triangular pockets correspond to a diamond shape of length $\varepsilon$ removed from each integer lattice point in the plane. Therefore, in the unfolded plane, the pockets become
\begin{align*}
    \mathcal{C} = C_{\varepsilon} + \mathbb{Z}^{2},
\end{align*}
where $C_{\varepsilon} = \{(x,y): |x|+|y| \le \varepsilon\}$. This formulation is the same as in \cite{AlkanLedoanZaha, BocaCobeliZaha2000, BocaGologanZaha, BocaGologanZaha2}, for example. The region of the unfolded plane, free of these pockets, is
\begin{align*}
    \Omega= \{x \in \mathbb{R}^{2}/2\mathbb{Z}^{2}: x \notin \mathcal{C} \},
\end{align*}
where the quotient by $2\Z^2$ keeps track of the orientation of the path. 
As $\varepsilon\rightarrow0$, upon reflections each possible trajectory can be transformed into a line segment starting from the origin and ending at an integer lattice point $(m,n)$. Moreover, we must have gcd$(m,n)=1$; otherwise, the trajectory terminates at the first integer point it reaches, namely $(m/\text{gcd}(m,n), n/\text{gcd}(m,n))$. Such points with coprime coordinates are called \textit{visible points}. Thus, there exists a one-to-one correspondence between a trajectory with rational slope starting from the origin and a visible point in $\R^2$.

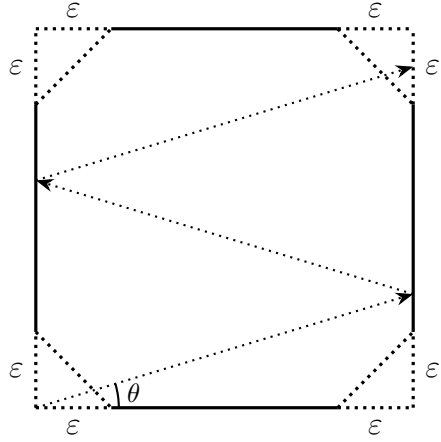
\begin{figure}[H]\label{fig:setting}
\centering
    \begin{tikzpicture}
    \draw[black, very thick](1,0)--(4,0)
    (0,1)--(0,4)
    (5,1)--(5,4)
    (1,5)--(4,5);
    \draw [dotted, very thick] (0,1) -- (1,0)
    (0,0) -- (1,0) node[midway, below]{$\varepsilon$}
    (0,0) -- (0,1) node[midway, left]{$\varepsilon$}
    (4,0) -- (5,0) node[midway, below]{$\varepsilon$}
    (5,0) -- (5,1) node[midway, right]{$\varepsilon$}
    (4,0) -- (5,1)
    (0,4) -- (1,5)
    (0,4) -- (0,5) node[midway, left]{$\varepsilon$}
    (0,5) -- (1,5) node[midway, above]{$\varepsilon$}
    (4,5) -- (5,5) node[midway, above]{$\varepsilon$}
    (4,5) -- (5,4)
    (5,4) -- (5,5) node[midway, right]{$\varepsilon$};
    
    \draw[dotted,thick, -{Stealth[scale = 1]}] (0,0) -- (5,1.5);
    \draw[dotted,thick, -{Stealth[scale = 1]}] (5,1.5) -- (0,3);
    \draw[dotted,thick, -{Stealth[scale = 1]}] (0,3)--(5,4.5);
    \coordinate (A) at (5,0);
    \coordinate (B) at (0,0);
    \coordinate (C) at (5,1.5);
    \draw[thick] (1.1,0) arc[start angle=0, end angle=16.7, radius=11mm];
    \node[font = \footnotesize] at (1.3,0.2) {$\theta$};
    \end{tikzpicture}
    \caption{The trajectory of the particle}
\end{figure}

The set of visible points is closely connected to Farey sequences. Indeed, a visible point $(q,a)$ in the first quadrant corresponds uniquely to the reduced fraction $a/q$, and the ordering of these fractions by slope is precisely captured by the Farey sequence $\mathcal{F}_Q$, which is defined as 
\begin{align}
\mathcal{F}_Q := \left\{\frac{a}{q}:1\le a\le q\le Q, \ (a,q)=1,\ Q\in \Z^{\ge 2}\right\}.\label{def:Farey sequence of order Q}
\end{align}
For our problem, the relevant order is  $Q=[1/\varepsilon ]$, where $[x]$ denotes the integer part of $x$. The following theorem provides an explicit count of self-intersection points in terms of pairs of consecutive Farey fractions $a/q< a'/q'$ in $\mathcal{F}_Q$.
\begin{thm}\label{thm:asymptotic of self-intersection count}
     Let $\mathcal{F}_Q$ be the Farey sequence of order $Q = [1/\varepsilon]$, and $\theta$ be the initial angle. Suppose~$a/q < a'/q'$ are consecutive fractions in $\mathcal{F}_Q$ such that $\tan\theta\in [a/q, \ a'/q')$. Denote by $I_\varepsilon(\theta)$ the number of self-intersection points before the trajectory enters a triangular pocket with size~$\varepsilon$ and initial angle $\theta$. Then,
     \begin{align*}
    I_{\varepsilon}(\theta) = 
        \begin{cases}
            \begin{cases}
            \frac{(a-1)(q-1)}{2} & \text{if $\tan(\theta) \in [\frac{a}{q}, \frac{a}{q-\varepsilon})$}\\
            \frac{(a-1)(q-1)}{2} & \text{if $\tan(\theta) \in [\frac{a}{q-\varepsilon}, \frac{a+\varepsilon}{q})$}\\
            \frac{(a'-1)(q'-1)}{2} & \text{if $\tan(\theta) \in [\frac{a+\varepsilon}{q}, \frac{a'}{q'})$}\\
            \end{cases} 
        & \text{if $q<q'$,}
        \\
            \begin{cases}
            \frac{(a-1)(q-1)}{2} & \text{if $\tan(\theta) \in [\frac{a}{q}, \frac{a'-\varepsilon}{q'})$}\\
            \frac{(a'-1)(q'-1)}{2} & \text{if $\tan(\theta) \in [\frac{a'-\varepsilon}{q'}, \frac{a'}{q'})$}\\
            \end{cases} 
            & \text{if $q>q'$ and $t_{S'}\le t_W$,}
        \\
        \begin{cases}
            \frac{(a-1)(q-1)}{2} & \text{if $\tan(\theta) \in [\frac{a}{q}, \frac{a}{q-\varepsilon})$}\\
            \frac{(a-1)(q-1)}{2} & \text{if $\tan(\theta) \in [\frac{a}{q-\varepsilon}, \frac{a'-\varepsilon}{q'})$}\\
            \frac{(a'-1)(q'-1)}{2} & \text{if $\tan(\theta) \in [\frac{a'-\varepsilon}{q'}, \frac{a'}{q'})$}\\
            \end{cases} 
            & \text{if $q>q'$ and $t_{S'}> t_W$},
        \end{cases}
    \end{align*}
   where $t_{S'} = \frac{a'-\varepsilon}{q'}$ and $t_W = \frac{a}{q-\varepsilon}$. 

    Moreover, the multiplicity of each self-intersection point is exactly one (meaning that the trajectory goes through each self-intersection point exactly twice).  
\end{thm}

Theorem \ref{thm:asymptotic of self-intersection count} fully characterizes the number of self-intersections for any irrational angle $\theta$, since $\tan\theta$ lies between some pair of consecutive Farey fractions of order $Q$. 

As corollaries of the above theorem, we provide asymptotic formulas for all 
$r$-th moments of $I_\varepsilon(\theta)$ for $r\ge 1$,  establish weak convergence  of a related probability measure, and derive an explicit density for the distribution of self-intersection points.

\begin{coro}\label{thm:integration over angles of self-intersection count}

For any interval $L \subseteq [0, \frac{\pi}{4}]$, any $\alpha, \delta > 0$ and $r \ge 1$, we have 
\begin{align*}
    &\varepsilon^{2r} \int \limits_{L} I_{\varepsilon}^{r}(\theta)d\theta = \frac{C_{r}}{2^{r}} \int \limits_{L} \tan^{r}(\omega) d\omega + O_{\alpha,r,\delta} \left( \varepsilon^{\frac{1}{2}-2 \alpha -\delta} +|L|\varepsilon^{\alpha}\right)
\end{align*} 
as $\varepsilon\rightarrow 0$, where
\begin{align}
    C_{r} 
    &= \frac{6}{r(2r+1)\pi^2}\left(\log 2 + \sum_{k=1}^{2r-1} \frac{1-2^{-k}}{k}\right)\label{eq:Cr final}\\
    &=\frac{6}{r(2r+1)\pi^2}[\log (2r-1)+\gamma] + O(r^{-3}).\notag
\end{align}
Here, $\gamma$ is the Euler–Mascheroni constant.

\begin{rem}
    Each $C_r$ in Corollary \ref{thm:integration over angles of self-intersection count} is a finite sum of terms that can be computed explicitly. For example, 
\begin{align*}
C_1 &= \frac{2}{\pi^2}\left(\log2+\frac12\right) &&\approx 0.24178,
&
C_2 &= \frac{3}{5\pi^2}\left(\log2+\frac76\right) &&\approx 0.11306,\\
C_3 &= \frac{2}{7\pi^2}\left(\log2+\frac{1531}{960}\right) &&\approx 0.06623,
&
C_4 &= \frac{1}{6\pi^2}\left(\log2+\frac{3193}{1680}\right) &&\approx 0.04380,\\
C_5 &= \frac{6}{55\pi^2}\left(\log2+\frac{1377977}{645120}\right) &&\approx 0.03127,
&
C_6 &= \frac{1}{13\pi^2}\left(\log2+\frac{16511491}{7096320}\right) &&\approx 0.02353.
\end{align*}

\end{rem}
\end{coro}
\begin{rem}
    The result for $\theta\in (\pi/4,\pi/2)$ follows immediately by symmetry. The reasoning follows similarly as in the proof of Theorem \ref{thm:asymptotic of self-intersection count}.
\end{rem}
\begin{coro}\label{coro: measure converge}
    Consider the probability measure $\mu^{L}_{\varepsilon}$ on $[0, \infty)$ defined by
    \begin{align*}
       \mu^{L}_{\varepsilon}(f)  = \frac{1}{|L|}\int \limits_{L} f(\varepsilon^{2} I_{\varepsilon}(\omega)) d\omega, \quad \text{where $f \in C_{c}([0, \infty))$}.
    \end{align*}
There exists a probability measure $\mu^{L}$ on $[0,1/2]$ such that 
    \begin{align*}
        \mu^{L}_{\varepsilon} \rightarrow \mu^{L} \quad \text{ as $\varepsilon \rightarrow 0^{+}$}.
    \end{align*}
    Moreover, the moments of $\mu^{L}$ are given by
    \begin{align*}
        \int \limits_{0}^{\infty} t^{n} d\mu^L(t) = \frac{C_{n}}{2^{n}|L|} \int \limits_{L} (\tan x)^n dx,
    \end{align*}
    where $n \in \mathbb{N}$ and $C_{n}$ is defined in Corollary \ref{thm:integration over angles of self-intersection count}.
\end{coro}

Corollary \ref{coro: measure converge} follows from Corollary \ref{thm:integration over angles of self-intersection count} similarly as \cite[Corollary 1.6]{BocaGologanZaha} follows from \cite[Theorem 1.5]{BocaGologanZaha}.  


Before stating the next corollary, we present the analog of \cite[Prop. 3.7]{BocaGologanZaha2}, the proof of which follows in the same way with some appropriate changes. We define
\begin{align}
    \widetilde{F}_{\varepsilon, J, Q}(t) := \left|\left\{\omega: \tan\omega \in J, I_{1/Q}(\omega) > \frac{t}{\varepsilon^{2}}\right\}\right|.
\end{align}

\begin{prop}\label{prop: local distribution} 
Take $\theta, \theta_{1} \in (0,1)$ and suppose $J = [\tan\omega_{0}, \tan\omega_{1}]$ with $J \subset [0,1]$ having the property that $|J| \asymp\varepsilon^{\theta}$. Then
\begin{align*}
    \widetilde{F}_{\varepsilon, J, Q}(t) = C_{J}H\Bigg(\sqrt{\frac{2t}{\tan(\omega_{0})}} \Bigg) + O\big(E_{\theta, \theta_{1},\delta}(\varepsilon)\big)
\end{align*}
uniformly for t in any compact subset $(0, \infty) \setminus \{1,2\}$, where
\begin{align*}
    C_{J} = \int \limits_{J} \frac{dt}{1+t^{2}} = \omega_{1} - \omega_{0}, \quad \quad E_{\theta, \theta_{1},\delta}(\varepsilon) = \varepsilon^{\min\left(\frac{3}{2}\theta,\ \frac{1}{2}-2\theta_{1}-\delta,\ \theta+\theta_{1}-\delta\right)}
\end{align*}
and
\begin{align*}
    H(t) = 
    \begin{cases}
        1 - \dfrac{12t}{\pi^2}, & 0 \le t \le \dfrac{1}{2}, \\[1.2ex]
        \dfrac{12}{\pi^2} \displaystyle\int_t^1 \dfrac{1-x}{x} \left(1 + \ln \dfrac{x}{1-x}\right) dx, & \dfrac{1}{2} \le t \le 1, \\[1.2ex]
        0, & t \ge 1.
    \end{cases}
\end{align*}
\end{prop}

\begin{coro}\label{thm: explicit density}
    Consider the repartition function
    \begin{align*}
        F_{\varepsilon}(t) := \frac{4}{\pi}\left|\{\omega \in (0,\pi/4] : \varepsilon^{2}I_{\varepsilon}(\omega) > t\}\right|.
    \end{align*}
    Then for each $t > 0$, the limit $F(t) := \lim\limits_{\varepsilon \rightarrow 0^{+}} F_{\varepsilon}$ exists. Specifically, one has
    \begin{align*}
        F_{\varepsilon}(\omega) = F(t) + O_\delta(\varepsilon^{1/10-\delta})
    \end{align*}
    with $\delta > 0$ and
    \begin{align*}
        F(t) = \frac{4}{\pi} \int \limits_{0}^{\frac{\pi}{4}}H\Bigg(\sqrt{\frac{2t}{\tan(\omega)}}\Bigg)d\omega 
    \end{align*}
\end{coro}
Corollary \ref{thm: explicit density} can be derived from Proposition \ref{prop: local distribution} similarly to how \cite[Theorem 1.1]{BocaGologanZaha2} was derived from \cite[Proposition 3.7]{BocaGologanZaha2}.
\begin{rem}
    Notice that the pole at $\omega=0$ in $F(t)$ in Corollary \ref{thm: explicit density}, which does not occur in \cite[Theorem 1.1]{BocaGologanZaha2}, does not cause any issues. For any fixed $t>0$, there exists $\upsilon>0$ such that $\tan\upsilon < 2t$, so for sufficiently small $\varepsilon>0$, we have $\tan\upsilon < 2t-\varepsilon$. By Theorem \ref{thm:asymptotic of self-intersection count}, for $0<\omega\le \upsilon$, we have 
    \[
    \varepsilon^2 I_\varepsilon (\omega)\le \frac{\tan \omega+\varepsilon}{2}<t, 
    \]
    so the contributions from $\omega\in (0,\upsilon)$ to $F(t)$ is $0$ for sufficiently small $\varepsilon$.
\end{rem}
\section{Properties of Farey Sequences and Kloosterman Sums}
In this section, we recall some known properties of Farey sequences and Kloosterman sums.

Let $Q \ge 2$ be an integer, and denote the Farey sequence of order $Q$, defined in \eqref{def:Farey sequence of order Q}, by $\mathcal{F}_{Q}$. Recall some basic features of consecutive elements in $\mathcal{F}_Q$, collected from \cite[Theorems 28, 30]{HardyWright}, \cite[Lemma 1]{Hall1970}, and \cite[Theorem 9.3]{LeVeque}.
\begin{lem}\label{lem:lemma 2.1}
  \begin{enumerate}
      \item If $a/q$ and $a'/q'$ are two successive terms in  $\mathcal{F}_Q$, then 
\[
a'q - aq' = 1\text{ and } q+q'>Q.
\]
\item If $1\le \max(q,q')\le Q<q+q'$, $a'q-aq'=1$ with $0\le a'<q'$ and $0\le a<q$, then $a/q$ and $a'/q'$ are consecutive elements in $\mathcal{F}_Q$.  
\item  If $q>q'$ then $a\ge a'$; if $q'>q$ then $a' > a$.
  \end{enumerate}  

\end{lem}
\begin{lem}\label{lem:property for asymptotic}
Let $a/q$ and $a'/q'$ be two successive terms in  $\mathcal{F}_Q$. Then,
\begin{align*}
    (a-1)^{r}(q-1)^{r} &= (aq)^{r} + O_{r}(a^{r-1}q^{r}),\\
(a'-1)^{r}(q'-1)^{r} &= (a'q')^{r} + O_{r}(a'^{r-1}q'^{r}),\\
\quad\left( \frac{a'}{q'}  \right)^{r} &= \left( \frac{a}{q} \right)^{r} + O_{r}\left(\frac{1}{Q} \right).
\end{align*}
\end{lem}
\begin{proof}
    The proof follows immediately by direct expansion.
\end{proof}

Additionally, we borrow the following lemmas from \cite{BocaCobeliZaha2000} and \cite{BocaGologanZaha}.

\begin{lem}\cite[Lemma 2.2]{BocaGologanZaha}\label{lem:zaha lemma 2.2}
    Assume that $q \ge 1$ is an integer, $\mathcal{I}$, $\mathcal{J}$ are intervals which contain at most $q$ integers, and $f: \mathcal{I} \times \mathcal{J} \rightarrow \mathbb{R}$ is a $C^{1}$ function. Then for any integer $T > 1$ one has
    \begin{align*}
        \sum_{\substack{a \in \mathcal{I}, b \in \mathcal{J} \\ ab \equiv 1 \bmod q}} f(a,b) = \frac{\varphi(q)}{q^{2}} \underset{\mathcal{I} \times \mathcal{J}}{\int\int}f(x,y) dx\ dy + E_{q,\mathcal{I}, \mathcal{J},f,T},
    \end{align*}
    where
    \begin{align*}
        E_{q,\mathcal{I},\mathcal{J},f,T}  \ll_{\delta} T^{2}q^{\frac{1}{2} + \delta}\|f\|_{\infty} + Tq^{\frac{3}{2}+\delta}\|Df\|_{\infty} + \frac{|\mathcal{I}||\mathcal{J}|\cdot\|Df\|_{\infty}}{T}
    \end{align*}
    for all $\delta > 0$. Here, $\|\cdot\|_{\infty}$ denotes the $L^{\infty}$ norm on $\mathcal{I} \times \mathcal{J}$.
\end{lem}

\begin{lem}\cite[Lemma 2.3]{BocaCobeliZaha2000}\label{lem:lemma 2.3}
    Suppose that $0 < a< b$ are real numbers, $q \in \mathbb{N}$ and f is piecewise $C^{1}$ on $[a,b]$. Then
    \begin{align*}
        \sum_{a < k \le b} \frac{\varphi(k)}{k}f(k) = \frac{6}{\pi^{2}}\int _{a}^{b}f + O\left(\log b\left(\|f\|_{\infty} + \int_{a}^{b}|f'|\right)\right).
    \end{align*}
\end{lem}

\section{The formula for the number of self-intersection points}

In this section, we begin with some preliminary results related to the square billiard with the triangular $\varepsilon$-pockets. In the end, we present the proof for Theorem \ref{thm:asymptotic of self-intersection count}, which gives the number of self-intersection points given any initial angle $\theta\in (0,\pi/2]$.

\begin{lem}\label{lem:unique intersection point}
    The multiplicity of each self-intersection point is exactly one, i.e., the particle can hit any point in $\Omega$ at most twice. 
\end{lem}

\begin{proof}
Fix an initial shooting angle $\theta \in [0, \frac{\pi}{2}]$. By the property of specular reflection, up to parallelism (i.e., ignoring the direction sign), there are exactly two cases for the trajectory of the particle along which the particle can move after each bounce: one with slope $\tan{\theta}$ and the other with slope $-\tan{\theta}$. Moreover, the angle between these two lines is $2\theta$, measured from the horizontal. Therefore, two consecutive line segments in a trajectory cannot be parallel, unless the initial angle is $90^\circ$. Taking the directions into account, there are four directed lines, namely 
\[\overrightarrow{\ell}
_{\tan\theta},\ \overleftarrow\ell_{\tan\theta},\ \overrightarrow{\ell}_{-\tan\theta},\ \overleftarrow{\ell}_{-\tan\theta}.\]

Suppose the trajectory has passed through the point $(x,y)\in\Omega$ Without loss of generality, we may assume the corresponding line is $\overrightarrow{\ell}
_{\tan\theta}$. Based on the above discussion, there are exactly four ways that the point $(x,y)$ could be re-encountered:

If the particle returns along $\overleftarrow\ell_{\tan\theta}$, then the trajectory would be reversed, eventually falling back into the starting point. This forces the trajectory to end at a visible point with both coordinates even, which is impossible because the two coordinates of any visible point must be coprime. See detailed discussion, for example, in \cite[Page 26]{AlkanLedoanZaha}. 

If the particle returns along $\overrightarrow{\ell}
_{\tan\theta}$, then the trajectory would be periodic. We claim that such period always contains the starting point, which leads to the same contradiction as in the first case. To see this, suppose we have two distinct trajectories that eventually have the same repeating pattern after the $n$-th bounce. Recall the discussion in the beginning, the segment in the ($n+1$)-th bounce must be parallel to that of ($n-1$)-th bounce. This forces the ($n-1$)-th bounce of the two distinct trajectories to be parallel. However, recall that up to parallelism, there are exactly two distinct lines. Since both the $n$-th and ($n-1$)-th bounces of the two trajectories are parallel, the initial angle $\theta$ is the same, which yields a unique trajectory.

Therefore, a self-intersection point can occur only if the particle passes along one of the two lines with slope $-\tan\theta$. For the particle to pass through $(x,y)$ a third time, it would again have to follow one of the four directed lines. But as argued above,  any such return would force a periodic or reversed trajectory, leading to a contradiction. Therefore, the particle can pass through any point at most twice, so the multiplicity of each self-intersection point is exactly one. 
\end{proof}

In the remainder of this section, we prove the first statement of Theorem \ref{thm:asymptotic of self-intersection count}. We first introduce several preliminary lemmas. 
\begin{lem}\label{lem:unique visible point}
    Let $a,q$ be positive integers such that gcd$(a,q) = 1$. Suppose the particle is shot from the origin with slope $\tan(\theta) = a/q$. There is exactly one visible point, namely $(q,a)$. Moreover, the trajectory does not enter a pocket before reaching $(q,a)$ if and only if $\varepsilon < \min(1/a,\ 1/q)$.
    \end{lem}

\begin{proof}
   The trajectory in the unfolded plane is given by the line $y=(a/q)x$. Recall that the plane also has $\varepsilon$-sized triangular pockets at the four corners of each unit square, forming an $\varepsilon$-diamond shape around each lattice point. 
   
   Consider a lattice point $(m,n)$ in the plane. The vertical distance between the point and the line is $|\frac{a}{q}m - n|$ and the horizontal distance is $|\frac{q}{a}n - m|$. The lattice point $(q,a)$ is visible if and only if the vertical and the horizontal distances are both greater than $\varepsilon$ for all $(m,n)$ up to $(q,a)$; otherwise, the trajectory ends in a pocket before $(q,a)$. Hence, for $(q,a)$ to be the only visible point, we must have
   \begin{align*}
       \min_{0 < m < q}\big|\frac{a}{q}m-n\big| > \varepsilon \textrm{ and } \min_{0 < n < a}\big|\frac{q}{a}n-m\big| > \varepsilon. 
   \end{align*}
Since $\min|(a/q)m-n| = 1/q$ and $\min|(q/a)n-m| = 1/a$, taking the smaller of the two concludes the proof.

\end{proof}

\begin{lem}\label{lem: terminated pocket}
     Let $a,q$ be positive integers such that gcd$(a,q) = 1$. Suppose the particle is shot from the origin with slope $\tan(\theta) = a/q$ and $\varepsilon < \min(1/a,1/q)$. Then the trajectory terminates in a corner pocket determined by the parity of $a$ and $q$: if $a,q$ are both odd, then the particle ends at the top-right corner pocket. If $a$ is odd and $q$ is even, then the particle ends at the top-left corner pocket. If $a$ is even and $q$ is odd, then the particle ends at the bottom-right pocket.
\end{lem}

\begin{proof}
    By Lemma \ref{lem:unique visible point}, the trajectory reaches $(q,a)$ first. Reflecting coordinates modulo $2$ maps the point into the two-by-two square positioned at the origin. The parity of $a$ and $q$ determines which quadrant it lies in, corresponding exactly to one corner pocket as listed. 
\end{proof}

\begin{coro}\label{coro:never initial pocket}
    Let $a,q$ be positive integers such that gcd$(a,q) = 1$. Suppose the particle is shot from the origin with slope $\tan(\theta) = a/q$ and $\varepsilon < \min(1/a,1/q)$. Then the trajectory never terminates in the initial bottom-left corner pocket. 
\end{coro}

\begin{proof}
    Since gcd$(a,q)=1$, at least one of $a$ and $q$ is odd. Therefore, the trajectory cannot terminate at the bottom-left pocket, which requires both to be even. 
\end{proof}


\begin{lem}\label{lem:formula for intersection points}
   Let $a,q$ be positive integers such that gcd$(a,q) = 1$. Suppose the particle is shot from the origin with slope $\tan(\theta) = a/q$ and $\varepsilon < \min(1/a,1/q)$.  Then the number of self-intersection points in the billard table is given by 
   \[
   I_\varepsilon\left(\frac{a}{q}\right) = \frac{(a-1)(q-1)}{2}.
   \]
\end{lem}

\begin{proof}
    With the given assumptions, the trajectory in the unfolded plane is given by $(x(t),y(t)) = (t,(a/q)t)$. By Lemmas \ref{lem:unique visible point}, \ref{lem: terminated pocket}, and Corollary \ref{coro:never initial pocket}, the unfolded trajectory reflects into the two-by-two square, containing the initial unit square centered at the origin.  

    Now consider a point $(x,y)$ inside the two-by-two square. In order to map such a point back into our initial unit square, we reflect the $x$-coordinate with respect to the vertical axis if $x \in [1,2]$, and the $y$-coordinate with respect to the horizontal axis if $y \in [1,2]$. Hence, given any $(x,y)$ in the two-by-two square, the mapping $(1-|1-x|),(1-|1-y|)$ maps the point into the initial unit square.

    Suppose $P_1(t_1) = (x_1,y_1)$ and $P_2(t_2) = (x_2,y_2)$ are points in the unfolded plane with $t_{1} < t_{2}$. Then $P_1(t_{1})$ intersects  $P_2(t_{2})$ if and only if
    \[\begin{cases}
        1-|1-x_{1}|\pmod 2 = 1-|1-x_{2}|\pmod2, \\
        1-|1-y_{1}|\pmod2 = 1-|1-y_{2}|\pmod2.
    \end{cases}\]

    Simplifying the absolute value gives the following cases:
    
    \textbf{Case 1:} 
    \begin{align}\label{case 1 step 1}
        \begin{cases}
            1-x_{1}\pmod2 = 1-x_{2}\pmod2, \\
        1-y_{1}\pmod2 = 1-y_{2}\pmod2.
        \end{cases}
    \end{align}

    Equation \eqref{case 1 step 1} is equivalent to the congruence 
    \begin{align*}
        \begin{cases}
            x_{1}-x_{2} \equiv  0 \pmod2,\\
            y_{1}-y_{2} \equiv  0 \pmod2.
        \end{cases}
    \end{align*}
Writing both $P_{1}$ and $P_{2}$ in terms of $t_{1}$ and $t_{2}$, we get $t_{1}-t_{2} = 2k$ for the first equation and $(a/q)(t_{1}-t_{2}) = 2h$ for the second equation with $k,h \in \mathbb{Z}$. Substituting the first equation into the second gives $ak-hq = 0$, which implies $k = q$ and $h = a$. However, from $t_{1}-t_{2} = 2k$, we have $t_{1} = 2q+t_{2}$, which is a contradiction since $0 \le t_{1} \le q$.\\

    \textbf{Case 2:} 
    \begin{align}\label{case 2 step 2}
        \begin{cases}
            1-x_{1}\pmod2 = -1+x_{2}\pmod2, \\
        1-y_{1}\pmod2 = -1+y_{2}\pmod2.
        \end{cases}
    \end{align}

    Equation \eqref{case 2 step 2} is equivalent to the congruence 
    \begin{align*}
        \begin{cases}
            x_{1}+x_{2} \equiv  0 \pmod2,\\
            y_{1}+y_{2} \equiv  0 \pmod2.
        \end{cases}
    \end{align*}

    Following the same reasoning as in case 1, we obtain a contradiction for case 2 as well.\\

    \textbf{Case 3:} 
    \begin{align}\label{case 3 step 3}
        \begin{cases}
            1-x_{1}\pmod2 = 1-x_{2}\pmod2, \\
        1-y_{1}\pmod2 = -1+y_{2}\pmod2.
        \end{cases}
    \end{align}

    Equation \eqref{case 3 step 3} is equivalent to the congruence equations given by
    \begin{align*}
        \begin{cases}
            x_{1}-x_{2} \equiv  0 \pmod2,\\
            y_{1}+y_{2} \equiv  0 \pmod2.
        \end{cases}
    \end{align*}
   Substituting $x_{1} = t_{1}, x_{2} = t_{2}$ and $y_{1} = (a/q)t_{1}, y_{2} =(a/q)t_{2}$, we obtain $t_{1}-t_{2} = 2k$ and $\frac{a}{q}(t_{1}+t_{2}) = 2h$ for some $k,h \in \mathbb{Z}$. Solving for $t_{1}$ and $t_{2}$ gives us the following:
    \begin{align*}
        \begin{cases}
            t_{1} = k+\frac{qh}{a},\\
            t_{2} = -k+\frac{qh}{a}.
        \end{cases}
    \end{align*}
    Excluding the start and pocket positions, since $0 < x_{1} < y_{2} < q$, we have $0 < t_{1} < t_{2} < q$. Therefore, the inequality becomes 
      \begin{align}
            0 < ak+qh < qh-ak < aq.\label{eq:double inequalities}
    \end{align}
Substituting $K=-k$, the inequalities in \eqref{eq:double inequalities} correspond exactly to the restrictions 
\[
K>0, \text{ and } \frac{a}{q}K<h<-\frac{a}{q}K+a.
\]
 Therefore, the integer solutions $(K,h)$ correspond exactly to the interior lattice points in the triangle with vertices $(0,0)$, $(0,a)$, and $(q/2,a/2)$. \\

    \textbf{Case 4:} 
    \begin{align}\label{case 4 step 4}
        \begin{cases}
            1-x_{1}\pmod2 = -1-x_{2}\pmod2, \\
        1-y_{1}\pmod2 = 1+y_{2}\pmod2.
        \end{cases}
    \end{align}
    This is equivalent to
    \begin{align*}
        \begin{cases}
            x_{1}+x_{2} \equiv  0 \pmod2,\\
            y_{1}-y_{2} \equiv  0 \pmod2.
        \end{cases}
    \end{align*}
By the same argument as in Case 3, we get
    \begin{align*}
        \begin{cases}
            t_{1} = k+\frac{qh}{a},\\
            t_{2} = k-\frac{qh}{a}.
        \end{cases}
    \end{align*}
 Hence we obtain the following inequalities:
      \begin{align*}
            0 < ak+qh < ak-qh < aq.
    \end{align*}
Similar to Case 3, the integer solutions correspond to the interior lattice points of the triangle with vertices $(0,0)$, $(q,0)$, and $(q/2,a/2)$. Now combining the two regions arising in Cases 3 and 4, we obtain precisely the interior of the triangle with vertices $(0,0)$,$(q,0)$, and $(0,a)$.
Applying Pick's Theorem to this triangle, we obtain that the number of interior lattice points is $(a-1)(q-1)/2$, which is exactly the number of self-intersection points. This concludes the proof.
\end{proof}

Now we are ready to prove Theorem \ref{thm:asymptotic of self-intersection count}.

\begin{proof}[Proof of Theorem \ref{thm:asymptotic of self-intersection count}]
Assume $\theta\in [0,\pi/4]$. Consider the trajectory of the particle on the unfolded board. Since the unfolded billiard table is invariant under reflections across horizontal and vertical lines, it suffices to consider the pocket configuration where the terminal pockets lie in the upper-left and lower-right directions. The other possible orientations can be obtained by symmetry. Let $O = (0,0)$, $A=(q,a)$, $A'=(q',a')$, $W = (q-\varepsilon,a)$, $W' = (q'-\varepsilon,a')$, $S'=(q',a'-\varepsilon)$, and $N=(q,a+\varepsilon)$. Then the triangular pocket adjacent to $A$, denoted by $P_A$, is the convex hull formed by vertices $A,W,N$. Similarly, denote the pocket adjacent to $(q',a')$ by $P_{A'}$, which is the convex hull formed by vertices $A', W', S'$. By \cite[Lemma 3.1]{BocaGologanZaha}, the unfolded path can hit only $P_A$ or $P_{A'}$, and no other earlier pockets. Without loss of generality, assume the unfolded path terminates in pocket $P_A$. For the rational slope $a/q$,  by Lemma \ref{lem:formula for intersection points},
\[
I_\varepsilon(a/q) = \frac{(a-1)(q-1)}{2}.
\]
We claim that the self-intersection point count for $\tan\theta$ is equal to the count for the rational slope $a/q$. Assume that there exists an additional self-intersection point. 
Since the trajectory with slope $a/q$ already accounts for all self-intersection points before $P_A$, this additional self-intersection point can occur only when the path passes over a lattice point in the interior of the triangle $\Delta OAW$. In other words, the path must pass through a lattice point corresponding to a rational slope $a''/q''$, where $a/q<a''/q''<\tan\theta$ and $q''\le Q$. However, by assumption, $a/q$ and $a'/q'$ are consecutive Farey fractions in $\mathcal{F}_Q$, so we conclude that no such lattice point can exist. Therefore,
\[
I_\varepsilon(\theta) = \frac{(a-1)(q-1)}{2}.
\]
It remains to prove which terminal pocket the particle falls in. We split the proof into three cases:
\begin{enumerate}
\item If $q<q'$:
Then one has $a<a'$, and therefore,
\[
t_A = \frac{a}{q}\le \frac{a}{q-\varepsilon}=t_W\le \frac{a+\varepsilon}{q}=t_N\le \frac{a'}{q'}=t_{A'}.
\]
If $\tan\theta\in [t_A, t_W)$, then the particle hits the horizontal side connecting $A$ and $W$ of $P_A$. Similarly, if $\tan\theta\in [t_W, t_N)$, then the particle hits the diagonal side of $P_A$. Therefore, in both cases the particle falls into pocket $P_A$. If $\tan\theta\in [t_N, t_{A'})$, the particle is above the top vertex $N$ of pocket $P_A$, and therefore it must hit pocket $P_{A'}$ instead. 
    \item If $q>q'$ and $t_{S'}\le t_W$: 
    If $\tan\theta\in [t_A,t_{S'})$, since $t_{S'}\le t_W$, the particle hits the horizontal side of pocket $P_A$. If $\tan\theta\in [t_{S'},t_{A'})$, the particle hits the vertical side of pocket $P_{A'}$.
    \item If $q>q'$ and $t_{S'}>t_W$: If $\tan\theta\in[t_A,t_W)$, then the particle hits the horizontal side of pocket $P_A$. If $\tan\theta\in[t_W, t_{S'})$, the particle hits the vertical side of pocket $P_A$. If $\tan\theta\in[ t_{S'}, t_{A'})$, then the particle hits the vertical side of pocket $P_{A'}$.
\end{enumerate}
This concludes all cases, and proves the result for $\theta\in[0,\pi/4]$.

Now suppose $\theta\in(\pi/4,\pi/2)$. Observe that 
\[
I_\varepsilon(\theta) = I_\varepsilon\left(\frac{\pi}{2}-\theta\right) 
\]
up to swapping the horizontal and vertical coordinates. Therefore, the corresponding adjacent terminal rational lattice point has coordinates swapped. This implies that the total count of self-intersection points is unchanged because 
\[
\frac{(a-1)(q-1)}{2} = \frac{(q-1)(a-1)}{2}. 
\]
Together with Lemma \ref{lem:unique intersection point}, this finishes the proof of Theorem \ref{thm:asymptotic of self-intersection count}.
\end{proof}

\section{The \texorpdfstring{$r$-th}{r-th} moment of self-intersection points}

In this section we discuss the corollaries presented in the introduction. The proofs follow similar arguments to those in \cite{BocaGologanZaha} and \cite{BocaGologanZaha2}. We therefore present only the detailed proof of Corollary \ref{thm:integration over angles of self-intersection count} here, which concerns the $r$-th moment of the number of self-intersection points, and leave the proofs of Corollary \ref{coro: measure converge}, Proposition \ref{prop: local distribution}, Corollary \ref{thm: explicit density} to the reader.

We begin by considering $\int_{L}I^{r}_{\varepsilon}(\theta)d \theta$, where  $L = [\alpha,\beta] \subseteq [0,\pi/4]$. Let $t_{1} = \tan(\alpha)$, $t_{2} = \tan(\beta)$, and $J = [t_{1},t_{2}] \subseteq [0,1]$. To set up, let the order of the Farey sequence be $Q =[1/\varepsilon]$.


    By Theorem \ref{thm:asymptotic of self-intersection count}, we can write the  integral over the self-intersection points by distinguishing cases by $q<q'$ or $q>q'$. Denote by $\mathcal{F}_{Q}^{<}$ the set of all consecutive pairs of Farey sequences $(a/q,a'/q')$ with $q<q'$, and by $\mathcal{F}_{Q}^{>}$ with $q>q'$. Moreover, we denote
\begin{align*}
    \sideset{^J}{}{\sum}_{a/q} := \sum_{\substack{(a/q,a'/q') \in \mathcal{F}_{Q}^{<} \\ a/q \in J}} \quad \text{and} \quad \sideset{}{^J}{\sum}_{a/q} := \sum_{\substack{(a/q,a'/q') \in \mathcal{F}_{Q}^{>} \\ a/q \in J}}.
\end{align*}

    The first reduction is similar to the methods used to prove \cite[Theorem 1.5]{BocaGologanZaha}, giving us
    \begin{align*}
         \int\limits_{L} I^{r}_{\varepsilon}(\theta) d\theta = A_{r,J,\varepsilon}+B_{r,J,\varepsilon}
    \end{align*}
    where 
    \begin{align*}
        A_{r,J,\varepsilon} = \sideset{^J}{}{\sum}_{a/q} \int_{w_{1}}^{w_{2}} I_{\varepsilon}^{r}(\theta)d\theta + \sideset{^J}{}{\sum}_{a/q} \int_{w_{2}}^{w_{3}} I_{\varepsilon}^{r}(\theta) d\theta
    \end{align*}
    and 
    \begin{align*}
        B_{r,J,\varepsilon} = \sideset{}{^J}{\sum}_{a/q} \int_{w_{1}}^{w_{4}} I_{\varepsilon}^{r}(\theta)d\theta  + \sideset{}{^J}{\sum}_{a/q} \int_{w_{4}}^{w_{3}} I_{\varepsilon}^{r}(\theta) d\theta
    \end{align*}

    with $w_{1}$, $w_{2}$, $w_{3}$, and $w_{4}$ given by 
    \begin{align*}
        w_{1} = \arctan\Big(\frac{a}{q}\Big), w_{2} = \arctan\Big(\frac{a+\varepsilon}{q}\Big), w_{3} = \arctan\Big(\frac{a'}{q'}\Big),\text{ and }\ w_{4} = \arctan\Big(\frac{a'-\varepsilon}{q'}\Big).
    \end{align*}

    Due to the symmetric nature of $A_{r,J,\varepsilon}$ and $B_{r,J,\varepsilon}$, it suffices to only consider  $A_{r,J,\varepsilon}$.
    Evaluating the inner integrals in  $A_{r,J,\varepsilon}$ gives  
    \begin{align}\label{eq:A}
        A_{r,J,\varepsilon} = \sideset{^J}{}{\sum}_{a/q} \left(\frac{(a-1)^{r}(q-1)^{r}}{2^{r}}\right)\left(\arctan\Big(\frac{a+\varepsilon}{q}\Big)- \arctan\Big(\frac{a}{q}\Big)\right) 
        \qquad
        \\
        \notag
        \qquad
        +\sideset{^J}{}{\sum}_{a/q} \left(\frac{(a'-1)^{r}(q'-1)^{r}}{2^{r}}\right)\left(\arctan\Big(\frac{a'}{q'}\Big)- \arctan\Big(\frac{a+\varepsilon}{q}\Big)\right).
    \end{align}

    Applying Lemma \ref{lem:lemma 2.1}, we have the inequality $q+q' \ge Q+1 > \frac{1}{\varepsilon}$. Together with Lemma \ref{lem:property for asymptotic}, upon simplication, we obtain
    \begin{align*}
        A_{r,J,\varepsilon} = \frac{1}{2^{r}} 
        \sideset{^J}{}{\sum}_{a/q}\frac{\big( \frac{a}{q} \big)^{r}}{1+ \big( \frac{a}{q} \big)^{2}} \left( \varepsilon q^{2r-1}
        + q'^{2r-1} \frac{1-\varepsilon q'}{q} \right) + E_{1}
        + O_{r} \left( Q^{2r-1} \right),
    \end{align*}
    where $E_{1} = O_{r}\big( \sideset{^J}{}{\sum}_{a/q} Q^{2r-3} \big) = O_{r}(Q^{2r-1})$, since $|\mathcal{F}_{Q}|= O(Q^{2})$. Therefore, we get  
    \begin{align}\label{eq:A final reduction}
        A_{r,J,\varepsilon} = \frac{1}{2^{r}} 
        \sideset{^J}{}{\sum}_{a/q}\frac{\big( \frac{a}{q} \big)^{r}}{1+ \big( \frac{a}{q} \big)^{2}} \left( \varepsilon q^{2r-1}
        + q'^{2r-1} \frac{1-\varepsilon q'}{q} \right)  
        + O_{r} \left( Q^{2r-1} \right).
    \end{align}
    
    Denote the inverse of $x$ modulo $q$ by $\overline{x}$ and $J_{q} := [(1-t_{2})q, (1-t_{1})q]$. By Lemma \ref{lem:lemma 2.1}, we have $a'q - aq' = 1$. Converting it into congruences, the main term of \eqref{eq:A final reduction} becomes
    \begin{align*}
        A_{r,J,\varepsilon}^{(1)} := \sum_{q= 1}^{Q}\sum_{\substack{\max\{Q-q,q\} < x \le Q \\ \overline{x} \in J_{q}}} f(x, \overline{x};q),
    \end{align*}
    where
    \begin{align*}
        f(x,y;q) :=& \frac{1}{2^{r}} 
        \frac{\big( \frac{q-y}{q} \big)^{r}}{1+ \big( \frac{q-y}{q} \big)^{2}} \left( \varepsilon q^{2r-1}
        + x^{2r-1} \frac{1-\varepsilon x}{q} \right) 
       =\frac{1}{2^{r}} 
        \frac{q^{r+1}(q-y)^{r}}{q^{2} + (q-y)^{2}} \left( \varepsilon
        + \frac{1-\varepsilon x}{q} \left(\frac{x}{q}\right)^{2r-1}   \right) ,
    \end{align*}
and the intervals $\mathcal{I}, \mathcal{J}$ are given by
\begin{align*}
     \mathcal{I} =& (\max\{Q-q,q\},Q], \quad\text{and }\ \mathcal{J} = J_{q}.
\end{align*}
    For $r\ge 1$, since $f$ is $C^1$ and $0\le 1-\varepsilon x\le \varepsilon(q+1)\ll q/Q$, we obtain the following estimates   
    \begin{align*}
        \|f\|_{\infty, \mathcal{I} \times \mathcal{J}} \ll_{r}& \frac{q^{r}(q-y)^{r}}{q^{2}+\left({q-y}\right)^{2}} \ll_{r} Q^{2r-2},
        \\
        \left\| \frac{\partial f}{\partial x}\right\|_{\infty, \mathcal{I} \times \mathcal{J}} \ll_{r}& \frac{1}{q^2} \frac{q^{r+1}q(q-y)^{r}}{q^{2}+\left({q-y}\right)^{2}} \ll_{r} \frac{Q^{2r-2}}{q},
        \\
       \textrm{and }\quad \left\| \frac{\partial f}{\partial y}\right\|_{\infty, \mathcal{I} \times \mathcal{J}} \ll_{r}& \frac{1}{q}\frac{q^{r+1}(q-y)^{r+1}}{(q^{2}+(q-y)^{2})^{2}} \ll_{r} \frac{Q^{2r-2}}{q}.
    \end{align*}
    Applying Lemma \ref{lem:zaha lemma 2.2} with $T = [Q^{\alpha}]$, and using the properties $|\mathcal{J}| = |J|q = Lq$ and $|\mathcal{I}| \le Q$, we simplify the inner sum of $A^{(1)}_{r,J,\varepsilon}$ as
    \begin{align*}
        &\frac{1}{2^{r}}\frac{\varphi(q)}{q^{2}} \int\limits_{\mathcal{I}} \left(\varepsilon + \frac{1-\varepsilon x}{q} \left(\frac{x}{q}\right)^{2r-1} \right)dx \int\limits_{\mathcal{J}} 
        \frac{q^{r+1}(q-y)^{r}}{q^{2} + (q-y)^{2}}dy
        \\
        \qquad
        &\quad+O_{\alpha,r,\delta} \left( Q^{2\alpha}q^{\frac{1}{2} + \delta} Q^{2r-2} + Q^{\alpha} q^{\frac{1}{2} + \delta} Q^{2r-2}+|L|Q^{-\alpha}qQ^{2r-2}\right)
        \\
        =& \frac{1}{2^{r}} \frac{\varphi(q)}{q}q^{2r-1} K_{r,q} Y_{r,L} + O_{\alpha,r,\delta} \left( Q^{2 \alpha +2r+\delta-\frac{3}{2}} +|L|Q^{-\alpha+2r-1}\right),
    \end{align*}
    where 
    \begin{align*}
        K_{r,q} &= \int\limits_{\max\{Q-q,q\}}^{Q}\left(\varepsilon + \frac{1-\varepsilon x}{q} \left(\frac{x}{q}\right)^{2r-1} \right)dx 
    \end{align*}
    and
    \begin{align*}
        Y_{r,L} & = \frac{1}{q^{2r-1}}\int\limits_{(1-t_{2})q}^{(1-t_{1})q} \frac{q^{r}(q-y)^{r}}{q^{2}+(q-y)^{2}}dy
        \\
        &= \frac{1}{q^{2r-1}}\int\limits_{t_{1}q}^{t_{2}q} \frac{q^{r+1}u^{r}}{q^{2}+u^{2}}du = \int\limits_{t_{1}}^{t_{2}} \frac{v^{r}}{1+v^{2}}dv 
       = \int\limits_{L} \tan^{r}(\omega) d\omega.
    \end{align*}
    As a result, for the inner sum of $A_{r,J,\varepsilon}^{(1)}$, we have    
    \begin{align}
        \sum_{\substack{\max\{Q-q,q\} < x \le Q \\ \overline{x} \in J_{q}}} f(x, \overline{x};q) = \frac{1}{2^{r}} \frac{\varphi(q)}{q} G_{r,q} Y_{r,L} + O_{\alpha,r,\delta} \left( Q^{2 \alpha +2r+\delta-\frac{3}{2}} +|L|Q^{-\alpha+2r-1}\right),
    \end{align}
    where 
    \begin{align*}
    G_{r,q} &:= \int\limits_{\max\{Q-q,q\}}^{Q}\left(\varepsilon q^{2r-1} + \frac{1-\varepsilon x}{q} x^{2r-1} \right)dx.
    \end{align*}

    Since $\|G_{r,q}\|_{\infty} = O_{r}(Q^{2r-1})$ and $\int_{1}^{Q} |G_{r,q}'|dq = O_{r}(Q^{2r-1})$, taking the outer sum and applying Lemma \ref{lem:lemma 2.3} give us
    \begin{align}
        A_{r,J,\varepsilon} &= \frac{Y_{r,L}}{2^{r}} \sum_{q=1}^{Q} \frac{\varphi(q)}{q}G_{r,q} + O_{\alpha,r,\delta} \left( Q^{2 \alpha +2r-\frac{1}{2}+\delta} +|L|Q^{-\alpha+2r}\right)
        \notag\\
        &= \frac{Y_{r,L}}{2^{r}\zeta(2)}\int \limits_{1}^{Q}G_{r,q} dq + O_{\alpha,r,\delta} \left( Q^{2 \alpha +2r-\frac{1}{2}+\delta} +|L|Q^{-\alpha+2r}\right).\label{eq:A(r,J,ep)}
    \end{align}
    
    Using the change of variables $q=Qx$ and the definition of $G_{r,q}$, we have
    \begin{align}\label{eq:messy integral}
        \int \limits_{1}^{Q} G_{r,q} dq = Q^{2r} \int \limits_{0}^{1} \bigg(x^{2r-1}( 1- \max\{1-x,x\}) + \frac{1- \max\{1-x,x\}^{2r}}{2rx}
        \\
        \notag
        -\frac{1- \max\{1-x,x\}^{2r+1}}{(2r+1)x} \bigg)dx + O_{r}(Q^{2r-1}).
    \end{align}

Define $C_{r}$ as 
    \begin{align}
        C_{r} &:= \frac{12}{\pi^{2}}\int \limits_{0}^{\frac{1}{2}}\left(x(x^{2r-1} + (1-x)^{2r-1}) + \frac{1-(1-x)^{2r}}{2rx(1-x)} - \frac{1-(1-x)^{2r+1}}{(2r+1)x(1-x)} \right).\label{def:Cr}
    \end{align}
    Then \eqref{eq:messy integral} becomes
    \begin{align}
        \int \limits_{1}^{Q} G_{r,q} dq = \frac{\pi^{2}Q^{2r}C_{r}}{12} + O_{r}(Q^{2r-1}).\label{eq:int G(r,q)dq}
    \end{align}

    Since $\zeta(2)=\pi^2/6$, \eqref{eq:A(r,J,ep)} and \eqref{eq:int G(r,q)dq} yield 
    \begin{align*}
        A_{r,J,\varepsilon} = \frac{Q^{2r}C_{r}}{2^{r+1}} \int \limits_{L} \tan^{r}(\omega) d\omega + O_{\alpha,r,\delta} \left( Q^{2 \alpha +2r-\frac{1}{2}+\delta} +|L|Q^{-\alpha+2r}\right).
    \end{align*}

Finally, by symmetry between $q$ and $q'$, we obtain $B_{r,J,\varepsilon}$ as
    \begin{align*}
        &B_{r,J,\varepsilon} = \frac{Q^{2r}C_{r}}{2^{r+1}} \int \limits_{L} \tan^{r}(\omega) d\omega + O_{\alpha,r,\delta} \left( Q^{2 \alpha +2r-\frac{1}{2}+\delta} +|L|Q^{-\alpha+2r}\right).   
    \end{align*}
    Summing $A_{r,J,\varepsilon}$ and $B_{r,J,\varepsilon}$ gives the asymptotic formula in Corollary \ref{thm:integration over angles of self-intersection count}.

    Observe that the expression for $C_r$ in \eqref{def:Cr} is equivalent to \eqref{eq:Cr final} after evaluating the integrals in \eqref{def:Cr}. Indeed, for the third integrand, we have
\[
\int \limits_{0}^{1/2} \frac{1-(1-x)^{2r}}{2rx(1-x)}dx = \frac{1}{2r}\int \limits_{0}^{1/2} \frac{1-(1-x)^{2r}}{x}dx + \frac{1}{2r}\int \limits_{0}^{1/2} \frac{1-(1-x)^{2r}}{(1-x)}dx .
\]
Subtitute $u=1-x$, we obtain
\[\frac{1}{2r}\int \limits_{0}^{1/2} \frac{1-(1-x)^{2r}}{x}dx + \frac{1}{2r}\int \limits_{0}^{1/2} \frac{1-(1-x)^{2r}}{(1-x)}dx = \frac{1}{2r}\int_{1/2}^1\frac{1-u^{2r}}{1-u}du+\frac{1}{2r}\int_{1/2}^1\frac{1-u^{2r}}{u}du.
\]
Using the geometric sum formula, 
\[
\frac{1}{2r}\int_{1/2}^1\frac{1-u^{2r}}{1-u}du = \frac{1}{2r}\int_{1/2}^1\sum_{k=0}^{2r-1} u^k du.
\]
Switching the order of integration and summation, we obtain
\[
\frac{1}{2r}\int_{1/2}^1\frac{1-u^{2r}}{1-u}du = \frac{1}{2r}\sum_{k=0}^{2r-1} \frac{1-2^{-(k+1)}}{k+1} = \frac{1}{2r}\sum_{k=1}^{2r} \frac{1-2^{-k}}{k}.
\]
The remaining integral is straightforward to evaluate, and together we have the contribution from the third integrand as
\[
\frac{1}{2r}\left(\sum_{k=1}^{2r} \frac{1-2^{-k}}{k}+\log 2-\frac{1-2^{-2r}}{2r}\right).
\]
Similar computation applies to the fourth integrand. This completes the proof of Corollary \ref{thm:integration over angles of self-intersection count}.

\end{document}